\documentclass[11pt]{article}

\usepackage{amsmath, amsthm, amssymb, amscd}
\usepackage{graphicx}
\usepackage{bm}
\usepackage{tikz}
\usepackage[title]{appendix}
\usepackage{comment}

\usepackage{hyperref}

\providecommand{\abs}[1]{\lvert #1 \rvert}
\providecommand{\norm}[1]{\lVert #1 \rVert}

\allowdisplaybreaks[4]
\numberwithin{equation}{section}

\swapnumbers

\numberwithin{equation}{section}
\newtheorem{thm}{Theorem}[section]
\newtheorem{lem}[thm]{Lemma}
\newtheorem{rem}[thm]{Remark}
\newtheorem{prop}[thm]{Proposition}

\newtheorem{exam}[thm]{Example}

\newcommand{\E}{\mathbb{E}}
\renewcommand{\P}{\mathbb{P}}

\providecommand{\eps}{\varepsilon}

\renewcommand{\theta}{\vartheta}
\renewcommand{\phi}{\varphi}

\newcommand{\MLE}{\operatorname{MLE}}

\newcommand{\KL}{\operatorname{KL}}

\newcommand{\B}{\mathbb{B}}
\newcommand{\R}{\mathbb{R}}
\newcommand{\N}{\mathbb{N}}

\newcommand{\mC}{\mathcal{C}}
\newcommand{\mF}{\mathcal{F}}

\newcommand{\mX}{\mathcal{X}}

\newcommand{\bbH}{\mathbb{H}}

\begin{document}

% "Title of the paper"
\title{Posterior contraction rates\\
 for support boundary recovery\footnote{We thank Marc Hoffmann, Richard Nickl and the members of the Bayes Club in Amsterdam for helpful comments. We are grateful to two anonymous referees and the associate editor for numerous detailed suggestions that significantly improved the article.
Financial support by the DFG through research unit FOR 1735 {\it Structural Inference in Statistics: Adaptation and Efficiency} is gratefully acknowledged. In addition, the second author was partially supported by an NWO TOP grant.
}}

\author{\parbox[t]{6cm}{\centering Markus Rei\ss\\Institute of Mathematics\\ Humboldt-Universit\"at zu Berlin\\ mreiss@math.hu-berlin.de}  \hspace{1cm} \parbox[t]{6cm}{\centering Johannes Schmidt-Hieber\\ Department of Applied Mathematics \\ University of Twente \\ a.j.schmidt-hieber@utwente.nl} }

%\date{This version: \today}
\date{}
\maketitle

\begin{abstract}
Given a sample of a Poisson point process with intensity $\lambda_f(x,y) = n \mathbf{1}(f(x) \leq y),$ we study recovery of the boundary function $f$ from a nonparametric Bayes perspective. Because of the irregularity of this model, the analysis is non-standard. We establish a general result for the posterior contraction rate with respect to the $L^1$-norm  based on entropy and one-sided small probability bounds.  From this, specific posterior contraction results are derived for  Gaussian process priors and priors based on  random wavelet series.
\medskip



\end{abstract}

% \paragraph{AMS 2010 Subject Classification:}
% Primary 62G10; secondary 62G15, 62G20.
%
% %62G10   	Hypothesis testing
% %62G15   	Tolerance and confidence regions
% %62G20   	Asymptotic properties
%
% \paragraph{Keywords:} Brownian motion; convexity; differential inequalities; ill-posed problems; mode detection; monotonicity; multiscale statistics; shape constraints.

\section{Introduction}

We consider a  support boundary detection model, where a Poisson point process (PPP) $N$ on $[0,1]\times \mathbb{R}$ is observed with intensity
\begin{align*}
	\lambda (x,y) = \lambda_f (x,y) = n \mathbf{1}(f(x)\leq y).
\end{align*}
The statistical task is to recover the unobserved lower boundary $f:[0,1] \rightarrow \mathbb{R}$ of the support of $\lambda$, see the  simulated data set in Figure \ref{Fig1}. This boundary detection model can be seen as a continuous analogue of the nonparametric regression model with discrete equidistant design and exponential errors, that is, we observe $Y_{i,n}= f(i/n) + \eps_{i,n},$ $i=1,\ldots,n,$ and $(\eps_{i,n})_i$ are i.i.d. exponential random variables, cf. \cite{meisterreiss2013, jirak2014}. As with the Gaussian white noise model for regular regression, we expect that our posterior contraction rates will transfer to this or even more general boundary regression models. The main structural point is that due to the one-sided error distribution, these models are not Hellinger differentiable and therefore irregular.

For classification problems, one often faces unbalanced designs and almost no uncertainty about the label classification. The most extreme case of correctly labeled training data and unbalanced design is if we only observe data from one class. In this case, we can still do binary classification if we additionally make assumptions on the distribution of the design. The support boundary model is an instance of such a scenario under the assumptions that the design in the observed class has been generated from a PPP and that the decision boundary is a function $x\mapsto f(x).$

\begin{figure}\label{Fig1}
\begin{center}
	\includegraphics[scale=0.65]{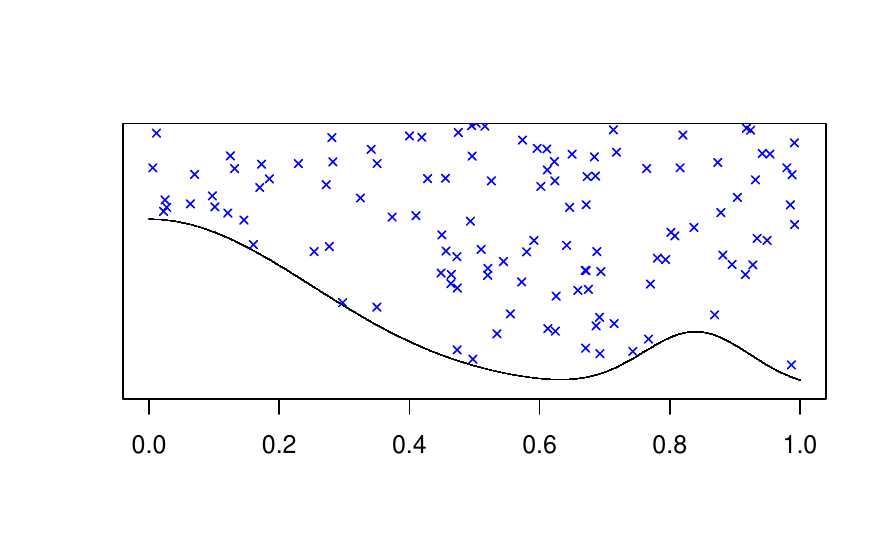}
\end{center}
\vspace{-1.6cm}
\caption{Simulated dataset (blue) and true boundary function (black).}
\end{figure}

Most of the nonparametric models that have been analysed from a frequentist Bayes point of view are asymptotically equivalent to a Gaussian shift experiment. Yet Poisson experiments form another important class of limit experiments \cite{LeCam1990}, whose statistical structure is  very different. The laws are not mutually absolutely continuous leading to a peculiar version of the Bayes formula and one-sided entropy conditions, subsequently. Moreover, the Hellinger distance is governed by the $L^1$-distance between the boundary functions in contrast to the $L^2$-theory in Gaussian shift models.

The goal of this article is to study posterior contraction for the support boundary detection model. We consider the $L^1$-distance as loss function, which is linked to the information geometry of the  model. Posterior contraction for the Hellinger loss is well-studied and can be reduced to conditions on the entropy of the parameter space and the small ball probability of the prior, cf. \cite{ggv, ghosal2007}. We derive a modification of this result which is applicable for the support boundary detection model. Related to that, we show the following surprising result: If the posterior is restricted to functions that lie below the true function, then posterior contraction follows already from the behaviour of the one-sided small ball prior probability. In this case no bound on the entropy is necessary. On the contrary, for functions which lie above the true function, we essentially only require the entropy bound.

Given the general contraction result, we apply this to concrete classes of priors. In a first step, we study Gaussian priors and derive an analogue of the result in \cite{vvvz} for the support boundary detection model. We then study posterior contraction for random wavelet series priors with independent but not necessarily Gaussian random coefficients. For these priors we derive a result on small ball probabilities, which is of independent interest. The corresponding contraction rates only match with the minimax estimation rates for one smoothness index. Below this critical smoothness the contraction rates can be improved if more heavy-tailed distributions on the wavelet coefficients are used. We also prove that truncated random wavelet series priors achieve the adaptive rates up to logarithmic factors. The companion paper \cite{reiss2018b} studies compound Poisson process priors for support boundary recovery. The focus of that article is on Bernstein-von Mises type theorems for function classes with increasing parameter dimension and frequentist coverage of credible sets.

Bayesian methods for irregular or boundary detection problems have attracted considerable attention especially because the MLE approach is often inefficient. \cite{chernozhukov2004} compares Bayes estimators with the MLE in a parametric model that is irregular. In \cite{bochkina2014} a Bernstein-von Mises theorem is derived for parameters which are on the boundary of the parameter space. The limit distribution consists in this case of Gaussian and exponentially distributed components. \cite{2012arXiv1210.6204K} considers posterior contraction around $\theta$ given i.i.d. observations from a class of nonparametric densities of the form $\eta(x-\theta)$ with $\eta(y)=0$ for $y<0$ and $\eta(y)>0$ for $y\geq 0.$ This can be viewed as a semiparametric, irregular model, where the nuisance parameter is the unknown distribution of the noise. For nonparametric models, \cite{Lo1982} considers Bayesian methods for Poisson point processes, but does not cover boundary detection. In \cite{LiGhoshal2015} a nonparametric Bayes approach is studied for detecting the boundary of an object in an image, assuming different distributions of the response variable inside and outside the object. This boundary detection model is regular and the likelihood ratios are always well-defined. The underlying information geometry is induced by the $L^1$-norm, similar to our PPP model, but there is no different treatment necessary for the posterior on functions below or above the true function.

The paper is structured as follows. In Section \ref{sec.general_results}, we derive a general result relating posterior contraction to entropy and small ball estimates. This result is then used in Section \ref{sec.GP_priors} to derive a criterion for posterior contraction under Gaussian priors. Section \ref{sec.Wavelet} studies wavelet expansion priors. Technicalities and proofs are deferred to an appendix.

{\it Notation.} We write $(x)_+ = \max(x,0)$ and denote the indicator function of a set $A$ by $\mathbf{1}_A = \mathbf{1}(\cdot \in A).$ For $p \in [1, \infty],$ $\|\cdot\|_p$ denotes the $L^p[0,1]$-norm. Inequalities for $L^1$-functions are assumed to hold almost everywhere. Let $\lfloor \beta\rfloor$ denote the largest integer strictly smaller than $\beta>0.$ The $\beta$-H\"older norm is $\|f \|_{\mC^\beta}:=\sum_{j=1}^{\lfloor \beta\rfloor}\|f^{(j)}\|_\infty + \sup_{x\neq y} |f^{(\lfloor \beta \rfloor)}(x) - f^{(\lfloor \beta \rfloor)}(y)|/|x-y|^{\beta -\lfloor \beta \rfloor}.$ We denote by $\mC^\beta(R)$ the class of functions $f$ on $[0,1]$ with $\|f\|_{\mC^\beta} \leq R.$ We further write $N = \sum_i \delta_{(X_i,Y_i)}$ for a random point measure on $[0,1] \times \mathbb{R}$ and often identify $N$ with its support points $(X_i, Y_i)_i.$ For two positive sequences $(a_n)_n, (b_n)_n$ we write $a_n \lesssim b_n$ if there is a constant $C$ such that $a_n \leq Cb_n$ for all $n.$ If $a_n \lesssim b_n$ and $b_n \lesssim a_n$ then we write $a_n \asymp b_n.$

\section{General results on posterior contraction rates}
\label{sec.general_results}

\subsection{Likelihood and Bayes formula}

Before stating the main result on posterior contraction, we  study the likelihood in the support boundary detection model. From that we  derive expressions for the information distances and a specific form of the Bayes formula.

Denote by $P_f=P_f^n$ the distribution of a PPP with intensity measure $\Lambda_f(B)=\int_B\lambda_f$ for Borel sets $B$ in $[0,1]\times\R$ with Lebesgue density $\lambda_f (x,y)= n \mathbf{1}(f(x) \leq y)$, where $f$ is some function in $L^1([0,1])$. The likelihood ratio $dP_f/dP_g$ is only defined for $g \leq f,$ otherwise $P_g$ does not dominate $P_f.$ The fact that the observation laws are not necessarily mutually absolutely continuous is a distinctive feature of support estimation problems and will play a major role in the analysis. Recall that for the Poisson point process $N$ its support points in $[0,1]\times \mathbb{R}$ are denoted by $(X_i, Y_i)_{i\ge 1}$.

\begin{lem}
\label{lem.com}
For $g \leq f$  and $f,g\in L^1([0,1])$, the likelihood ratio has the explicit form
\begin{align}
	\frac{dP_f}{dP_g} = \exp\Big(n \int_0^1 (f -g) (x) \, dx\Big) \cdot \mathbf{1}(\forall i : f(X_i)\leq Y_i).
\label{eq.RadonNikodym}
\end{align}
\end{lem}
The information geometry of the model is driven by the $L^1([0,1])$-norm. Indeed, the Hellinger affinity is $\rho(P_f, P_g) = \int \sqrt{dP_f dP_g} = \exp(-\tfrac n2 \|	f-g\|_1).$ This implies for the squared Hellinger distance
\begin{align*}
	H^2(P_f, P_g)
	&= 2- 2 \rho(P_f, P_g) = 2 -2 \exp\big(-\tfrac n2 \|	f-g\|_1\big)
	\leq n \|f-g\|_1, \ \  \forall f, g\in L^1([0,1]).
\end{align*}
Similarly, the Kullback-Leibler divergence satisfies $KL(P_f, P_g)=n \|f-g\|_1$ if $g\leq f$ and $KL(P_f, P_g)=\infty$ otherwise.

Since the likelihood requires the support boundaries to be in $L^1([0,1])$, we consider as priors  distributions $\Pi$ of stochastic processes $(X_t)_{t\in [0,1]}$  on a Polish space $(\Theta,d)$ equipped with its Borel $\sigma$-algebra, which embeds continuously into $L^1([0,1])$.  We aim for a Bayes formula of the form
\begin{align}
	\Pi(B | N) = \frac{\int_B  \frac{dP_f}{dP_{f_0}} (N) \, d \Pi(f)}{\int_\Theta  \frac{dP_f}{dP_{f_0}}(N) \, d \Pi(f)}.
	\label{eq.Bayes_target}
\end{align}
Since in the boundary detection model the likelihood ratio does  not exist in general, the formula has to be modified. The next result provides  a Bayes formula under the frequentist assumption that the data are generated under $P_{f_0}.$

\begin{lem}
\label{lem.Bayes_formula}
For $f_0\in L^1([0,1])$, a prior $\Pi$ on the Polish space $\Theta$  with $\Pi(\{f\in\Theta:f\ge f_0\})>0$ and a Borel set $B\subset\Theta$ we have an explicit Bayes formula under the law $P_{f_0}$:
\begin{align*}
	\Pi(B | N) = \frac{\int_B e^{n\int f} \mathbf{1} (\forall i : f(X_i) \leq Y_i) \, d\Pi(f)}{\int_\Theta  e^{n\int f} \mathbf{1}(\forall i : f(X_i) \leq Y_i) \, d\Pi(f)}
	= \frac{\int_B  e^{-n\int (f_0-f)_+} \frac{dP_{f\vee f_0}}{dP_{f_0}} (N) \, d\Pi(f)}{\int_\Theta  e^{-n\int (f_0-f)_+} \frac{dP_{f\vee f_0}}{dP_{f_0}} (N) \, d\Pi(f)}\quad P_{f_0}\text{-a.s.}
\end{align*}
\end{lem}

The right-hand side is well-defined since $dP_{f\vee f_0}/dP_{f_0}$ exists and $\Pi(\{f\in\Theta:f\ge f_0\})>0$ implies that $P_{f_0}$-almost surely the denominator does not vanish. Compared to \eqref{eq.Bayes_target}, the likelihood ratios are reweighted in the Bayes formula by a factor $e^{-n\int (f_0-f)_+}.$ In particular, for $f\leq f_0$ the integrands are equal to the deterministic values $e^{-n \int (f_0-f)}.$

%It is reasonable to assume that the true function $f_0$ lies below the upper boundary $F.$ Given a general prior $\Pi$ it makes therefore sense to replace it by the prior distribution conditioned to the set $\{f : f\leq F\},$ that is, $\Pi(A \, | \, f\leq F) = \Pi(A\cap \{f: f\leq F\})/\Pi(f\leq F).$ Consequently, we may assume that $\Pi(f: f\leq F)=1.$

\subsection{Main results}

We start by stating the main theorem, which reduces posterior contraction to conditions on the entropy and small ball probabilities. The result is an analogue of the general contraction theorems in \cite{ggv, ghosal2007}. Denote by $N(\eps, \mF, d)$ the $\eps$-covering number of a metric space $\mF$ with respect to the metric $d.$

\begin{thm}
\label{thm.main_ub}
If for some $\Theta_n \subset \Theta,$ some rate $\eps_n \rightarrow 0$ and constants $C, C',C''\geq 1$, $A>0$

\begin{tabular}{ll}
 (i) & \quad
   $N(\eps_n, \Theta_n, \|\cdot\|_\infty) \leq C'' e^{C'n\eps_n};$ \\[0.3cm]
  (ii) & \quad $\Pi(f: \|f-f_0\|_1 \leq A\eps_n, f\leq f_0) \geq e^{-Cn\eps_n};$ \\[0.3cm]
  (iii) & \quad $\Pi(\Theta_n^c)\leq C''e^{-(C+A+1) n \eps_n},$ \\
\end{tabular}

then there exists a positive constant $M$ such that
\begin{align*}
	&E_{f_0}\big[\Pi\big( f : \| f-f_0\|_1 \geq M\eps_n | N \big) \big]  \leq 3C'' e^{- n\eps_n}.
\end{align*}
Condition (i) can be relaxed to any of the conditions of Proposition \ref{prop.suffcond}.
\end{thm}

 %The first condition is given in terms of sup-norm covering numbers, but can be relaxed to bracketing entropy with respect to the $\|\cdot \|_1$-norm.
In condition \textit{(ii)} we need a lower bound on the one-sided small ball probabilities. Applying triangle inequality and $\|\cdot\|_1\le\|\cdot\|_\infty$, a stronger version of \textit{(ii)}, which  is often easier to verify, is given by
\begin{align}
	\textit{(ii)':} \ \ \Pi(f: \|f+A\eps_n/2-f_0\|_\infty \leq A\eps_n/2) \geq e^{-Cn\eps_n}.
	\label{eq.condii_sup_norm}
\end{align}
The proof of the theorem is deferred to the appendix, yet main intermediate results are presented here. It will be convenient to establish posterior contraction for $\int (f_0-f)_+$ and $\int (f-f_0)_+$ separately. Surprisingly, for posterior contraction with respect to $\int (f_0-f)_+$ we only need the small ball estimate of the prior probability, but no bound on the entropy. In contrast,  posterior contraction for $\int (f-f_0)_+$ only requires that \textit{(i)} and \textit{(iii)} of Theorem \ref{thm.main_ub} hold.

\begin{prop}
\label{prop.post_contr1}
If for some constants $C, A>0,$
\begin{align*}
\Pi(f: \norm{f_0-f}_1 \leq A\eps_n, f\leq f_0) \geq e^{-Cn\eps_n},
\end{align*}
then
\begin{align*}
	&E_{f_0}\Big[\Pi\Big( f :  \int(f_0-f)_+ \geq (1+A+C)\eps_n \Big | N \Big) \Big]  \leq e^{- n\eps_n}.
\end{align*}
\end{prop}

%{\color{red} Try to interpret for us / not the general reader: The small ball condition in the proposition is in particular satisfied if (noting $1+A+C>A$)
%\begin{align*}
%{\color{red}	\Pi(f\leq f_0\,|\,\int(f_0-f)_+ \leq A\eps_n) \geq e^{-Cn\eps_n}\frac{1-\Pi(\int(f_0-f)_+ \leq A\eps_n)}{\Pi(\int(f_0-f)_+ \leq A\eps_n)}.}
%\end{align*}
% Usually, the prior probability $\Pi(\int(f_0-f)_+ \leq A\eps_n)$ is bounded away from zero and then we ask that under the prior among all functions $f$ which cover at most an area $A\eps_n$ between $f$ and $f\vee f_0$ the proportion of $f$  below $f_0$ has at least the order $e^{-Cn\eps_n}$. In the case $\Pi( f : \int(f_0-f)_+\le A\eps_n)=1$ the condition is always satisfied and the contraction result holds trivially because the prior does not assign any mass to the set.  }

The one-sided small ball probability can be viewed as a prior mass condition on a Kullback-Leibler ball in view of $\{f: \KL(P_{f_0},P_f) \leq A\eps_n n \} =\{f: \int(f_0-f) \leq A\eps_n, f\leq f_0\}.$ To establish posterior contraction with respect to the loss $\int (f- f_0)_+,$ we need to understand the testing theory in the boundary detection model, which is non-standard due to the lack of absolute continuity in general. The Neyman-Pearson test $\phi = \mathbf{1}(dP_g/dP_{f\wedge g} \geq  dP_f/dP_{f\wedge g})$ behaves well for testing $f$ against $g$:
\begin{align*}
	E_f[\phi] + E_g[1-\phi]
	& = \int \Big(\frac{dP_f}{dP_{f\wedge g}} \wedge \frac{dP_g}{dP_{f\wedge g}} \Big) \, dP_{f\wedge g}
	\leq \rho(P_f,P_g) =e^{-\frac{n}{2}\|f-g\|_1}.
\end{align*}
Robustness with respect to the $L^1$-distance (i.e., Hellinger-distance), however, in the sense that for some $\alpha, \beta >0,$ and all $n$
\begin{align*}
	E_f[\phi] + \sup_{h : \|h-g\|_1 \leq \alpha \|f-g\|_1}E_h[1-\phi] \leq e^{- \beta n \| f-g\|_1}
\end{align*}
holds, is violated:
if $f \le g,$ we have $\phi={\bf 1}(\forall i: g(X_i)\le Y_i)$ and thus
$E_h[1-\phi]  = 1-e^{-n \int (g-h)_+}$,
which for general $h$ is much larger than $e^{- \beta n \| f-g\|_1}$. Under the additional assumption $h\ge g$, however, the type II error vanishes completely and we find for $f\le g$
\begin{align*}
	E_f[\phi] \leq e^{-\frac{n}{2}\|f-g\|_1} \quad \text{and} \quad \sup_{h\geq g} E_h[1-\phi] =0.
\end{align*}
To control the posterior, it is therefore natural to use one-sided bracketing entropy. Consider a subset $\mF$ of $L^1([0,1])$.
The one-sided bracketing number $N_{[}(\delta, \mF)$ is the smallest number $M$ of functions $\ell_1,\ldots,\ell_M\in L^1([0,1])$ such that for any $f\in \mF$ there exists $j\in \{1,\ldots, M\}$ with $\ell_j\leq f$ and $\int(f-\ell_j)\leq \delta.$
For some function $f_0$ and integer $n$ consider the separation quantity
\[S_{[}( n,\mF, f_0)=\inf_{(\ell_j)_{j\in J}}\sum_{j\in J} e^{-n \int(\ell_j-f_0)_+}\in[0,\infty],\]
where the infimum is taken over (not necessarily finite) subsets $J$ of the integers and  functions $(\ell_j)_{j\in J}\subset L^1([0,1])$ such that for any $f\in \mF$ there exists $j\in J$ with $\ell_j\leq f$. In both definitions the functions $\ell_j$ are not required to be in $\mF.$

In view of the next result, the quantity $S_[$, which can be seen as a weighted covering number, is the natural complexity measure for $\Theta$.

\begin{prop}
\label{prop.post_contr2}
If  $\Pi(f:  f \leq f_0) >0,$ then for any Borel set $B \subseteq \Theta$
\begin{align*}
	E_{f_0}\big[\Pi\big( f\in B  \big | N \big) \big]  \leq  S_{[} \big( n,B,f_0\big).
\end{align*}
\end{prop}

Notice that the right-hand side does not depend on the prior. Weighted covering numbers might be small even for non-compact parameter spaces and have been used before in nonparametric Bayes theory, cf. \cite{hoffmann2015}, Section 4. For many specific problems, covering or bracketing numbers are sufficient and we can further upper bound the right-hand side in Proposition \ref{prop.post_contr2}:

\begin{prop}\label{prop.suffcond}
Work under the assumption of Proposition \ref{prop.post_contr2}. If $C\geq 1,$ then
\[  E_{f_0}\Big[\Pi\Big( f\in \Theta_n : \int(f-f_0)_+ \geq 4C \eps_n  \Big | N \Big) \Big]\le C'' e^{-n\eps_n}\]
holds for $C''\ge 1$ under any of the following conditions:

\begin{tabular}{ll}
 (i) &
   $S_{[}( n,\{f\in\Theta_n:\int(f-f_0)_+\ge 4C \eps_n\},f_0)\le C''e^{-n\eps_n}$; \\[0.3cm]
  (ii) & $N_{[}\big(2C\eps_n, \Theta_n\big) \leq C'' e^{Cn \eps_n}$; \\[0.3cm]
  (iii) & $N(C\eps_n, \Theta_n, \|\cdot\|_\infty) \leq C'' e^{C n\eps_n}.$ \\
\end{tabular}

\end{prop}

With these propositions at hand we can easily derive Theorem \ref{thm.main_ub} in the appendix.

We can avoid the entropy condition if we control instead the risk of an estimator. Indeed, for a loss function $\ell$ the inequality
\[\inf_\phi \Big(E_{\theta_0}[\phi]+ \sup_{\theta \in \Theta : \ell(\theta, \theta_0) \geq 2\eps} E_\theta [1-\phi]\Big) \leq 2\inf_{\widehat \theta} \sup_{\theta \in \Theta}  P_\theta (\ell(\widehat \theta, \theta) \geq \eps)
\]
 follows by studying the test $\phi = \mathbf{1} (\ell(\widehat \theta, \theta) \geq \eps)$ given an estimator $\widehat \theta.$ If the nonparametric MLE for $f$ exists, we have a particularly simple relation in the support boundary detection model between posterior contraction of $\int(f-f_0)_+$ and the excess probability of the MLE. The following lemma holds even without any conditions on the prior.

\begin{lem}\label{LemBayesMLE}
Assume that $\Theta_n \subseteq \Theta$ contains $f_0$  and is closed under maxima, that is, if $f,g \in \Theta_n$, then $f\vee g \in \Theta_n.$ If the maximum likelihood estimator $\widehat f^{\MLE}$, based on the parameter space $\Theta_n$, exists, then
\begin{align}
	&E_{f_0}\Big[\Pi\Big( f\in \Theta_n : \int(f-f_0)_+ > \eps_n\Big | N \Big) \Big]
	\leq 	P_{f_0} \Big( \int(\widehat f^{\MLE}- f_0)_+> \eps_n\Big).
	\label{eq.MLE_risk_ineq}
\end{align}
\end{lem}

As in the proof of Proposition \ref{prop.post_contr2} the upper bound is independent of the prior. It is well-known that posterior contraction with rate $\eps_n$ implies existence of a frequentist estimator with rate of convergence $\eps_n,$ cf. Theorem 2.5 in \cite{ggv}. Inequality \eqref{eq.MLE_risk_ineq} shows that also the other direction may hold, namely that convergence of an estimator implies posterior contraction with the same rate. Regarding the assumptions, a sufficient condition for the existence of the MLE is that $\Theta$ is closed under arbitrary maxima: $f_i\in\Theta, i\in I \Rightarrow\bigvee_{i\in I} f_i\in\Theta$, see the discussion in \cite{reiss2014}. Examples of function spaces which are closed under the maximum are H\"older balls, monotone functions and convex functions.

\section{Gaussian process priors}
\label{sec.GP_priors}

A common choice for nonparametric Bayes methods is to pick the distribution of a Gaussian process as prior probability measure. Given a Gaussian process prior $\Pi,$ the seminal work in \cite{vvvz} relates posterior contraction to the small ball prior probability and approximation properties in the reproducing kernel Hilbert space (RKHS) generated by $\Pi.$ The following result adapts Theorem 2.1 in \cite{vvvz} to our setting.

\begin{thm}
\label{thm.ub_for_GPs}
Consider as prior $\Pi$ the distribution of a Gaussian process $X$ with sample paths in the space $(\mC[0,1], \|\cdot\|_\infty).$ Write $\|\cdot\|_{\bbH}$ for the RKHS-norm induced by the covariance operator of $X.$ If $\eps_n\geq n^{-1}$ and for all $n$
\begin{align}
	\inf_{h: \|h+ 2\eps_n-f_0\|_\infty \leq \eps_n} \|h\|_{\bbH}^2 - \log \P (\|X\|_\infty \leq \eps_n)\leq n \eps_n,
	\label{eq.RKHS_contraint}
\end{align}
then there exists a constant $M$ such that for all $n$
\begin{align*}
	&E_{f_0}\big[\Pi\big( f : \| f-f_0\|_1 \geq M\eps_n | N \big) \big]  \leq 3e^{-n\eps_n}.
\end{align*}
\end{thm}

If the infimum in the theorem is taken over the empty set, the left hand side in \eqref{eq.RKHS_contraint} is defined as $+\infty.$ Condition \eqref{eq.RKHS_contraint} is slightly different compared to (1.2) and (1.3) in \cite{vvvz}. As a bound we have $n\eps_n$ instead of $n\eps_n^2$ and in the RKHS part there is an extra term $2\eps_n$ which accounts for the one-sided prior mass condition in Theorem \ref{thm.main_ub}.

As the left-hand side of \eqref{eq.RKHS_contraint} has been studied for many classes of Gaussian processes, it is easy to obtain the corresponding contraction rates as a consequence of Theorem \ref{thm.ub_for_GPs}. For the main examples in \cite{vvvz} condition \eqref{eq.RKHS_contraint} becomes $\eps_n^{-1/\alpha}\lesssim n\eps_n$ and we obtain the optimal posterior contraction rate $n^{-\alpha/(\alpha+1)}$ for $f_0\in \mC^\alpha(R)$. We give three concrete examples:

\begin{exam}\mbox{}
\begin{enumerate}
\item Brownian motion. As prior we consider the law of the process $(X_0+W_t)_{t\in [0,1]}$ with a Brownian motion $W$ and an independent standard normal random variable $X_0$. Let $f_0\in\mC^\beta(R)$. Arguing as in \cite{vvvz}, Section 4.1, we find for the corresponding RKHS norm $\|h\|_{\bbH}^2 = \|h'\|_2^2 + h(0)^2$ and $\inf_{h: \|h+2\eps_n-f_0\|_\infty \leq \eps_n} \|h\|_{\bbH}^2 \lesssim \eps_n^{2-2/\beta}$ as well as for the small ball probabilities $P(\|X\|_\infty \leq \eps_n ) \geq  \P( |X_0| \leq \eps_n/2) \P(\|W\|_\infty \leq \eps_n/2) \gtrsim \eps_n e^{-C/\eps_n^2}$. The closure of $\bbH$ is $\mC[0,1]$ and \eqref{eq.RKHS_contraint} becomes
\begin{align}
	\eps_n^{2-2/\beta}+C\eps_n^{-2}+\log(\eps_n^{-1})\lesssim n\eps_n.
	\label{eq.BM_rate_eq}
\end{align}
Minimizing in $\eps_n$ yields the $L^1$-contraction rate
\begin{align}
	\begin{cases}
	n^{-\frac {\beta}{2-\beta}}, & \text{for} \ \beta \leq 1/2,\\
	n^{-\frac 13}, & \text{for} \ \beta \geq 1/2.
	\end{cases}
	\label{eq.BM_rates}
\end{align}
This coincides with the minimax rate $n^{-\beta/(\beta+1)}$ if $\beta = 1/2.$ For $\beta >1/2,$ we do not gain anymore in the contraction rate by imposing more smoothness on the signal. For $\beta< 1/2$ the rate is slower than the minimax rate.

It is instructive to compare the contraction rates to the ones obtained for regular models. The equivalent of \eqref{eq.BM_rate_eq} in regular models is $\eps_n^{2-2/\beta}+C\eps_n^{-2}+\log(\eps_n^{-1})\lesssim n\eps_n^2,$ see again \cite{vvvz}, Section 4.1. In this case, we obtain under a Brownian motion prior the contraction rates  $n^{-1/4}$ for $\beta\geq 1/2$ and $n^{-\beta/2}$  for $\beta<1/2,$ which are always slower than in \eqref{eq.BM_rates}.

\item Riemann-Liouville process. Similar to Brownian motion,   the Riemann-Liouville process $R_t^\alpha$  with parameter $\alpha >0$ starts at zero in zero and also the derivatives (if they exist) vanish at zero. The Riemann-Liouville process with random derivatives at zero is given by
\begin{align*}
	X_t = \sum_{k=0}^{\lceil \alpha\rceil} Z_k t^k + R_t^\alpha,
\end{align*}
for $\lceil \alpha\rceil$ the smallest integer strictly larger than $\alpha$ and independent $Z_1,\ldots, Z_{\lceil \alpha\rceil}\sim \mathcal{N}(0,1)$ which are also independent of $(R_t^\alpha)_t.$ From Theorem 4.3 in \cite{vvvz},  we find that \eqref{eq.RKHS_contraint} becomes $\eps_n^{-1/\alpha}\lesssim n\eps_n$ leading to the posterior contraction rate $n^{-\alpha/(\alpha+1)}.$

\item Fractional Brownian motion. For the Hurst index $\alpha \in (0,1)$  Theorem 4.4 in \cite{vvvz} yields that condition \eqref{eq.RKHS_contraint} for a fractional Brownian motion prior becomes $\eps_n^{-1/\alpha}\lesssim n\eps_n$ resulting again in the optimal posterior contraction rate $n^{-\alpha/(\alpha+1)}.$
\end{enumerate}
\end{exam}

\section{Wavelet expansion priors}\label{sec.Wavelet}

Series expansions provide another natural way to construct priors on function spaces. We study process priors $(X_t)_{t\in [0,1]}$ which admit an expansion in a wavelet basis $(\psi_{jk})$:
\begin{align}
	X_t = \sum_{j,k} d_{j,k}\xi_{j,k}\psi_{j,k}(t) \ \ \ \text{in} \ L^2[0,1].
	\label{eq.prior_wavelet}
\end{align}
Here, $d_{j,k}$ are real numbers and $\xi_{j,k}$ are i.i.d. random variables with Lebesgue density $f_\xi.$ As a prior on the function $f$ this means that each wavelet coefficient of $f$ is drawn independently from the distribution of $d_{j,k} \xi_{j,k}.$ For convenience, we restrict ourselves in this section to $s$-regular, boundary corrected and compactly supported wavelet bases $(\psi_{jk})$ in $L^2([0,1])$ as constructed in Section 4 of \cite{Cohen1993}.

Wavelet expansion priors have been studied in different nonparametric models with  uniform random variables $\xi_{j,k}$, cf. \cite{Gine2011, Ray2013}.  Moreover, \cite{wang1997} derives bounds on the small ball probabilities of Gaussian processes of the form \eqref{eq.prior_wavelet}.
%Whereas there are many results on small ball probabilities of Gaussian processes, relatively little is known for non-Gaussian processes. To allow for a few larger coefficients, it might be useful to use  a distribution of $\xi_{j,k}$ with heavier tails.
Below, we derive posterior contraction rates for a class of distributions $\xi_{j,k}.$  To start with, we prove the following general lower bound on small ball probabilities, which is of independent interest.

\begin{lem}
\label{lem.small_ball_prob_around_h}
Assume \eqref{eq.prior_wavelet} with a symmetric and unimodal density $f_\xi$ and $|d_{j,k}|\asymp 2^{-\frac j2 (2\alpha+1)}$ for some $\alpha>0.$ Suppose further that there exists a constant $\delta>0$ such that
\begin{align*}
	\E\big[|\xi_{j,k}|^{(1+\delta)/\alpha}\big] < \infty.
\end{align*}
Then for all $\beta\in(0,s]$, $R>0$ there exists a constant $D>0$ such that
\begin{align*}
	\inf_{h \in \mC^\beta(R)}\P\big(\|X-h\|_\infty \leq \eps \big)
	\geq f_\xi\big(D\eps^{-(\alpha -\beta)_+/\beta}\big)^{D\eps^{-1/(\alpha \wedge \beta)}} \quad \text{for all} \ \ 0< \eps \leq 1.
\end{align*}
\end{lem}

For $\beta \geq \alpha$ the lower bound has the form $C^{- \eps^{-1/\alpha}}$ with $C=f_\xi(D)^{-D}$. For $\beta < \alpha$ the lower bound depends on the tails of the distribution: heavier tails lead to larger lower bounds on the small ball probabilities. The fastest contraction rate that can be obtained using the small ball estimate in Lemma \ref{lem.small_ball_prob_around_h} and Theorem \ref{thm.main_ub} is $\eps_n= n^{-\beta/(1+\beta)}$, which is the solution of the equation $\eps_n^{-1/\beta} = n\eps_n.$

%This is the optimal rate if the smoothness of the signal matches the smoothness of the prior, that is $\beta=\alpha.$ If $\beta>\alpha$ then this rate is suboptimal. Nevertheless we can still get adaptation over a subinterval of $(0,\alpha].$

\begin{thm}
\label{thm.contr_rates_for_wav_priors}
Consider the process in \eqref{eq.prior_wavelet} as prior with a symmetric and unimodal density $f_\xi$ and $|d_{j,k}|\asymp 2^{-\frac j2 (2\alpha+1)}$ for some  $\alpha>0.$ Suppose  $f_\xi(x) \leq \gamma^{-1}e^{-\gamma |x|^q}$ for some $q>\alpha^{-1}$, some (sufficiently small) $\gamma > 0$ and all $x\in \R$. Fix $\beta\in(0,s]$, $R>0$. For any sequence $\eps_n \rightarrow 0$, satisfying
\begin{equation}\label{eq.decay_cond}
- \log f_\xi\Big(D\eps_n^{-\frac{(\alpha -\beta)_+}{\beta}}\Big)\lesssim	n \eps_n^{\frac {1+(\alpha \wedge \beta)}{\alpha \wedge \beta}},
\end{equation}
there exist positive constants $M$ and $c$ such that for all $n$
\begin{align*}
	&\sup_{f_0 \in \mC^\beta (R)}E_{f_0}\big[\Pi\big( f : \| f-f_0\|_1 \geq M\eps_n | N \big) \big]  \leq e^{-c n\eps_n}.
\end{align*}
In the case $q\in ((2\alpha)^{-1},\alpha^{-1}]$ the result remains true   under the additional assumption $\eps_n\gtrsim n^{-(2\alpha-1/q)/(2\alpha-1/q+1)+\delta}$ for some $\delta>0$.
\end{thm}

%The small ball estimate requires  that $|d_{j,k}| \gtrsim 2^{-\frac j2 (2\alpha+1)}$ as otherwise $X$ puts in these directions insufficient mass to larger coefficients.

The proof is based on verifying the conditions of Theorem \ref{thm.main_ub}. Since on high resolution levels more prior mass is assigned to large wavelet coefficients, heavy tails can also lead to a larger bias. In the proof, this is reflected in the choice of the set $\Theta_n$ which is taken to be a Besov $B^\alpha_{p,\infty}$-ball, where the $p$ depends on $q$ and $\alpha.$ The control of uniform entropy of this Besov space induces then the assumption $q>(2\alpha)^{-1}$. If the series coefficients are Gaussian, $q=2$ and the condition will be $\alpha>1/4.$ Surprisingly, the van der Vaart-van Zanten approach for Gaussian process priors does not require such a condition, see Theorem \ref{thm.ub_for_GPs}. In this case, condition $(iii)$ of Theorem \ref{thm.main_ub} is controlled via Borell's inequality, which allows to choose a set $\Theta_n$ with a better control of the high-frequencies avoiding any additional assumption. It is not clear to us whether  the condition $q>(2\alpha)^{-1}$ for non-Gaussian priors can be avoided. In Remark \ref{rem.on_bias} below, the conditions for $q$ and the popular choice of uniform priors for the coefficients (corresponding to $q=\infty$) are discussed further.

One of the consequences of Theorem \ref{thm.contr_rates_for_wav_priors} is that the posterior contracts faster in the regime $\beta< \alpha$ if heavier-tailed distributions are used. This is illustrated by the following specific example. Consider the wavelet expansion prior with density $f_\xi(x)\asymp e^{-\gamma\abs{x}^q}$ for some $\gamma,q>0$ and all $x\in\R$. Then condition \eqref{eq.decay_cond} reads $\eps_n^{-q(\alpha-\beta)_+/\beta}\lesssim n\eps_n^{(1+\alpha\wedge\beta)/\alpha\wedge\beta}$. For $\beta\ge\alpha>1/q$ we thus obtain the contraction rate $\eps_n\asymp n^{-\alpha/(\alpha+1)}$ which is minimax optimal for $\beta=\alpha$. In the case $\alpha>\beta\vee 1/q$ the contraction rate becomes
\[\eps_n= n^{-\beta/(\beta+q(\alpha-\beta)+1)}.\]
Hence, the contraction rate becomes faster for smaller $q,$ or equivalently, more heavy-tailed distributions for $(\xi_{jk})$. Notice, however, the constraint $q>1/\alpha$ for this result. For smaller $q$, down to $1/(2\alpha)$, we still have posterior consistency and for all $\beta\ge\alpha(1+(\alpha q)^{-1})(1-(2\alpha q)^{-1})$ we obtain  the rate $n^{-(2\alpha-1/q)/(2\alpha+1-1/q)}$, up to an arbitrarily small increase in the exponent. We give two concrete applications:
\begin{exam}
\label{exam.small_ball}
\renewcommand{\labelenumi}{(\alph{enumi})}
\begin{enumerate}
\item If $\xi_{j,k} \sim \mathcal{N}(0,1)$, then Lemma \ref{lem.small_ball_prob_around_h} yields for a sufficiently large constant $C$
$$\inf_{h \in \mC^\beta(R)}\P\big(\|X-h\|_\infty \leq \eps \big) \geq \exp \big(-C\eps^{- (\frac{1+2\alpha -2\beta}{\beta} \vee \frac 1{\alpha})}\big).$$ For $\alpha = 1/2,$ the bound becomes $\exp (-C\eps^{- 2(\frac{1 -\beta}{\beta} \vee 1)})$, which is the same as for the Brownian motion prior. Theorem \ref{thm.contr_rates_for_wav_priors} with $q=2$ yields for $\alpha>1/2$ the posterior contraction rate
\begin{align*}
	\eps_n = n^{-\frac{\beta \wedge \alpha}{2\alpha -\beta\wedge\alpha+1}}.
\end{align*}
For $\beta=\alpha$ this the minimax optimal rate $n^{-\alpha/(\alpha+1)}$.
In case $\alpha>1/4$ we still have posterior consistency, but with a slower rate.

\item If $\xi_{j,k}$ follows a Laplace (double-exponential) distribution, we obtain
$$\inf_{h \in \mC^\beta(R)}\P\big(\|X-h\|_\infty \leq \eps \big) \geq \exp \big(-C\eps^{- (\frac{1+\alpha -\beta}{\beta} \vee \frac 1{\alpha})}\big).$$
The posterior contraction rate becomes $\eps_n = n^{-\frac{\beta \wedge \alpha}{1+ \alpha}}$ if $\alpha>1$ which improves the rate in (a) for the case $\beta < \alpha,$ but relies on a stronger constraint on $\alpha$. For $\beta=\alpha$ we achieve the minimax optimal rate $n^{-\alpha/(\alpha+1)}$.
Posterior consistency is still guaranteed whenever $\alpha>1/2$.
\end{enumerate}
\end{exam}

We can also obtain a fully adaptive result (up to $\log n$ factors) using a random truncation of the wavelet expansion prior. The prior can be realized via a hierarchical construction. In a first step, we draw the maximal resolution level $J$ from a distribution satisfying
\[P(J=j)\propto \exp(-Bj2^{j})\]
for some constant $B>0$. Given $J,$ generate
\begin{align}
	X_t=\sum_{j \leq J, \, k}  \xi_{j,k}\psi_{j,k}(t)
	\label{eq.trunc_prior}
\end{align}
with $\psi_{j,k}$ as in \eqref{eq.prior_wavelet} and  $(\xi_{j,k})_{j,k}$ an  i.i.d. sequence of random variables with positive and continuous Lebesgue density $f_\xi$. In this prior the regularization is induced by the truncation of the wavelet series and compared with \eqref{eq.prior_wavelet} we can set $d_{j,k}=1.$

\begin{lem}
\label{lem.small_ball_prob_around_h_trunc_prior}
Consider the random truncation prior \eqref{eq.trunc_prior}. For $\beta \in (0, s]$, $R>0$  there exists a constant $D>0$ such that
\begin{align*}
	\inf_{h \in \mC^\beta(R)}\P\big(\|X-h\|_\infty\leq \eps \big)
	\geq \eps^{D\eps^{-1/\beta}} \quad \text{for all} \ \ 0< \eps \leq 1.
\end{align*}
\end{lem}

\begin{thm}
\label{thm.contr_rates_for_trunc_wav_priors}
Consider the random truncation prior \eqref{eq.trunc_prior}. Suppose  $f_\xi(x) \leq \gamma^{-1}e^{-\gamma |x|^q}$ for  some $q,\gamma > 0$ and all $x\in \R$ and fix $\beta\in(0,s]$, $R>0$. Then there exist constants $M$ and $c$ such that for all $n$
\begin{align*}
	&\sup_{f_0 \in \mC^\beta (R)} E_{f_0}\big[\Pi\big( f : \| f-f_0\|_1 \geq M\eps_n | N \big) \big]  \leq e^{-c n\eps_n}
\end{align*}
with
\begin{align*}
	\eps_n = \Big( \frac{\log n}{n} \Big)^{\frac{\beta}{\beta+1}}.
\end{align*}
\end{thm}

\appendix

\section{Proofs for Section \ref{sec.general_results}}
\begin{proof}[Proof of Lemma \ref{lem.com}]
The general change of measure formula for two Poisson point processes (PPPs) on $\mathcal X$ with finite intensity measures $\Lambda_1 \ll \Lambda_2$ is given by
\begin{align}
	\frac{dQ_{\Lambda_1}}{dQ_{\Lambda_2}}((X_i,Y_i)_{i\ge 1})
 &= \exp \Big( \sum_{i\ge 1} \log \Big(\frac{d\Lambda_1}{d\Lambda_2}(X_i,Y_i) \Big)
	- \Lambda_1(\mX)+ \Lambda_2(\mX) \Big)
	\label{eq.LR_PPPs}
\end{align}
where $\log 0:=-\infty$, $\exp(-\infty):=0$ and $(X_i,Y_i)$ denote the point locations of the PPP,
cf. \cite{kutoyants1998}, Theorem 1.3. Notice that $P_f$ and $P_g$ have infinite intensity. We therefore apply the following decomposition first. For $h\in L^1([0,1])$ split the state space $[0,1]\times\R$ into $H^-=\{(x,y)\,|\,x\in[0,1],y\in(-\infty,h(x))\}$, $H^+=\{(x,y)\,|\,x\in[0,1],y\in[h(x),\infty)\}$. Then by independence of the PPP on disjoint sets we may write $P_f=P_{f,H^-}\otimes P_{f,H^+}$ where generally $P_{f,H}$ denotes the law of the PPP with intensity $\lambda_{f,H}(x,y)=n{\bf 1}(f(x)\le y)$ on the set $H$. If $\Lambda_{f,H^-}$ denotes the intensity measure of $P_{f,H^-},$ we have $\Lambda_{f,H^-}(H^-)=n\int_0^1 (h-f)_+.$ For $h\ge f\ge g$ we obtain
\[ P_f=P_{f,H^-}\otimes P_{f,H^+}=P_{f,H^-}\otimes P_{h,H^+},\quad P_g=P_{g,H^-}\otimes P_{g,H^+}=P_{g,H^-}\otimes P_{h,H^+},\]
 remarking that $P_{f,H^+}=P_{g,H^+}=P_{h,H^+}$ are PPPs with intensities equal $n$ on $H^+$. Since on $H^-$ the intensity measures are finite, we derive
\begin{align*}
	\frac{dP_f}{dP_g}((X_i,Y_i)_{i\ge 1}) &=\frac{dP_{f,H^-}}{dP_{g,H^-}}(\{(X_i,Y_i)\,|\,i\ge 1\}\cap H^-) \frac{dP_{h,H^+}}{dP_{h,H^+}}(\{(X_i,Y_i)\,|\,i\ge 1\}\cap H^+)\\
	&= \exp\Big(\sum_{i: Y_i  < h(X_i) } \log\Big(\frac{n{\bf 1}(Y_i\ge f(X_i))}{n{\bf 1}(Y_i\ge g(X_i))}\Big)-n\int(h-f)+n\int(h-g)\Big)\\
	&= e^{n \int (f-g)} \mathbf{1}(\forall i : Y_i\ge f(X_i)),
\end{align*}
where we used that the argument of the logarithm is $P_g$-a.s. one or zero and the latter happens if $Y_i<f(X_i)$ for some $i$.
\end{proof}

\begin{proof}[Proof of Lemma \ref{lem.Bayes_formula}]
We first construct a dominating measure.
Let $f_0\in L^1([0,1])$ be fixed. Consider a PPP on $[0,1]\times \R$ with a strictly positive intensity $\lambda^*:[0,1]\times\R\to (0,\infty)$ satisfying
\[\lambda^* (x,y)= n\text{ for }x\in[0,1],\, y\ge f_0(x)\text{ and }\Lambda^*(\{y< f_0(x)\}) =\int_0^1\int_{-\infty}^{f_0(x)}\lambda^*(x,y)dydx<\infty,\]
and denote by $P^*$ its distribution. Let $H^{-}, H^{+}$ be as in the proof for Lemma \ref{lem.com} with $h=f\vee f_0.$ Observe that $\Lambda^*(H^{-})= \Lambda^*(\{y< f_0(x)\})+n\int_0^1 (f-f_0)_+.$

As in the proof for Lemma \ref{lem.com} we decompose $P^*=P^*_{H^-}\otimes P^*_{H^+}$ and $P_f=P_{f,H^-}\otimes P_{f,H^+},$ where $P^*_{H^-}$ and $P^*_{H^+}$ denote the restrictions of $P^*$ to $H^{-}$ and $H^{+}.$ Then $P^*_{H^+}=P_{f,H^+}$ because both intensities equal $n$ on $H^+$. Using \eqref{eq.LR_PPPs},
\begin{align*}
	&\frac{dP_f}{dP^*}((X_i,Y_i)_{i\ge 1}) = \frac{dP_{f,H^-}}{dP^*_{H^-}}(\{(X_i,Y_i)\,|\,i\ge 1\}\cap H^-)\\
	&= \exp\Big( \sum_{i: Y_i < f\vee f_0(X_i)} \log\Big(\frac{n}{\lambda^*(X_i,Y_i)}\Big)+ n\int(f-f_0) + \Lambda^*(\{y\le f_0(x)\})\Big)\mathbf{1}(\forall i : f(X_i) \leq Y_i),
\end{align*}
arguing as for Lemma \ref{lem.com}. Now, note  $P_{f_0}\ll P^*$ and  $Y_i\ge f_0(X_i)$  $P_{f_0}$-a.s. such that
\[ \frac{dP_f}{dP^*}((X_i,Y_i)_{i\ge 1}) =  \exp\Big( n\int(f-f_0) + \Lambda^*(\{y\le f_0(x)\})\Big)\mathbf{1}(\forall i : f(X_i) \leq Y_i)\quad P_{f_0}\text{-a.s.}
\]
Since $\Pi$ is defined on a Polish space and $\Pi(\{f\in\Theta:f\ge f_0\})>0$, the posterior is well-defined (cf. \cite{ghoshal2017}, Section 1.3) and
\begin{align*}
	\Pi(B | (X_i,Y_i)_{i\ge 1})
	&=
	\frac{\int_B \frac{dP_f}{dP^*}((X_i,Y_i)_{i\ge 1}) d\Pi(f)}{\int \frac{dP_f}{dP^*}((X_i,Y_i)_{i\ge 1}) d\Pi(f)} \\
	&=
	\frac{\int_B e^{n\int (f-f_0)}\mathbf{1}(\forall i : f(X_i) \leq Y_i) d\Pi(f) }{\int e^{n\int (f-f_0)}\mathbf{1}(\forall i : f(X_i) \leq Y_i) d\Pi(f) }, \quad P_{f_0}\text{-a.s.}
\end{align*}
Under $P_{f_0}$ we have $\mathbf{1} (\forall i : f(X_i) \leq Y_i) = \mathbf{1} (\forall i : f\vee f_0(X_i) \leq Y_i)$ a.s. and \eqref{eq.RadonNikodym} yields
\begin{align}
	H(f) := e^{n\int (f - f_0)} \mathbf{1}(\forall i : f\vee f_0(X_i) \leq Y_i) = e^{-n\int (f_0-f)_+} \frac{dP_{f\vee f_0}}{dP_{f_0}}( (X_i,Y_i)_{i\ge 1}
),
	\label{eq.Hdef}
\end{align}
which completes the proof.
\end{proof}

\begin{proof}[Proof of Proposition \ref{prop.post_contr1}]
Consider $H(f)$ from \eqref{eq.Hdef}. By Lemma \ref{lem.Bayes_formula} and the assumption on $\Pi$ we have under $P_{f_0}$
\begin{align}
\begin{split}
		\Pi(B | N)= \frac{\int_B  H(f) \, d\Pi(f)}{\int H(f) \, d\Pi(f)}
	&\leq e^{An\eps_n} \frac{\int_B  H(f) \, d\Pi(f)}
	{  \Pi(f: \|f-f_0\|_{1}\leq A\eps_n, f\leq f_0)} \\
&\le e^{(A+C)n\eps_n} \int_B  H(f) \, d\Pi(f),
\end{split}
\label{eq.lb_denominator}
\end{align}
where we used  $f(X_i)\leq Y_i$ $P_{f_0}$-a.s. for all $f\leq f_0.$  With  $E_{f_0}[H(f)] = e^{-n\int (f_0-f)_+}\le e^{-(1+A+C)n\eps_n}$ for all $f$ with $\int (f_0-f)_+\ge (1+A+C)\eps_n$, we obtain the result.
\end{proof}

\begin{proof}[Proof of Proposition \ref{prop.post_contr2}]
For functions $(\ell_j)_{j\in J}$, eligible in the definition of $S_{[}(n,B,f_0)$, consider the test  $\phi_n={\bf 1}(\exists j\forall i:\ell_j(X_i)\le Y_i)$. This test satisfies under the hypothesis $f_0$
\[ P_{f_0}(\phi_n=1)\le \sum_{j\in J} P_{f_0}(\forall i:\ell_j(X_i)\le Y_i)=  \sum_{j\in J}e^{-n\int(\ell_j-f_0)_+}.
\]
By assumption and $\sigma$-continuity of $\Pi$, there exist $R>0$ and $\delta>0$ such that $\Pi(f: \int f\ge -R, f\leq f_0)\geq \delta.$ Thus, we use formula \eqref{eq.lb_denominator} and bound the posterior by
\begin{align*}
	\Pi(B | N)
	\leq \phi_n
	+ \frac{\int_B H(f) (1-\phi_n) d\Pi(f)}{\int H(f) d\Pi(f)}
	\leq \phi_n + \delta^{-1} e^{nR +n \int f_0} \int_B H(f) (1-\phi_n) d\Pi(f).
\end{align*}
Since for $f\in B$ there is an $\ell_j\le f$, we infer
\[H(f) (1-\phi_n)=e^{n\int (f  - f_0)}{\bf 1}(\forall i:f(X_i)\le Y_i){\bf 1}(\forall j\exists i:\ell_j(X_i)>Y_i)=0.\]
 Therefore,
\begin{align*}
	E_{f_0}\big[\Pi(B | N)\big]
	\leq E_{f_0}\big[\phi_n\big]
	\leq   \sum_{j\in J}e^{-n\int(\ell_j-f_0)_+}
\end{align*}
and the claim follows by taking the infimum over all possible $(\ell_j)$.
\end{proof}

%\begin{prop}
%\label{prop.lb}
%If for some constant $A>0$ and a sequence of sets $\Theta_n$
%\begin{align*}
%	\frac{\Pi(\Theta_n^c) }{\Pi\big(f: \|f-f_0\|_1 \leq A\eps_n, f\leq f_0 \big)} \leq C'e^{-(A+1) n \eps_n},
%\end{align*}
%then
%\begin{align*}
%	E_{f_0}\big[\Pi\big( \Theta_n^c | N\big) \big] \leq C' e^{-n \eps_n }.
%\end{align*}
%\end{prop}

\begin{proof}[Proof of Proposition \ref{prop.suffcond}]
 For any $\delta>0,$ the one-sided bracketing entropy $N_{[}(\delta, \Theta_n)$ provides us with  functions $(\ell_j)_{j\in J}$ that can be used to bound $S_{[}( n,\Theta_n,f_0).$ Together with the inequality $-\int (\ell_j -f_0)_+ \leq - \int (f-f_0)_+ + \int (f-\ell_j)_+,$  this implies
\begin{align*}
\textstyle S_{[}\big( n,\{f\in\Theta_n:\int(f-f_0)_+\ge 4C\eps_n \},f_0\big)
& \leq e^{- 2Cn\eps_n} N_{[}\big(2C\eps_n, \Theta_n\big)\\
&\leq e^{- 2Cn\eps_n} N\big(C\eps_n, \Theta_n, \|\cdot \|_\infty\big).
 \end{align*}
It remains to apply Proposition \ref{prop.post_contr2}.
\end{proof}

\begin{proof}[Proof of Theorem \ref{thm.main_ub}]
By Proposition \ref{prop.post_contr1} and Proposition \ref{prop.suffcond} it remains to show  $E_{f_0}[\Pi ( \Theta_n^c | N) ] \leq C'' e^{-n \eps_n }.$ By \eqref{eq.lb_denominator} and $E_{f_0}[H(f)] \leq 1$ we obtain
\begin{align*}
	E_{f_0}\big[\Pi\big( \Theta_n^c | N\big)\big] \leq \frac{e^{An  \eps_n}  \Pi(\Theta_n^c)}{\Pi(f: \|f-f_0\|_1\leq A \eps_n , f\leq f_0)}.
\end{align*}
The claim thus follows from conditions \textit{(ii)} and \textit{(iii)}.
\end{proof}

\begin{proof}[Proof of Lemma \ref{LemBayesMLE}]
The key observation is that we can  restrict the posterior to $\{f\leq \widehat f^{\MLE}\}$ because on the complement the likelihood is zero. To see this, note that $\forall i:f(X_i)\le Y_i$  implies $f\le \widehat f^{\MLE}$ because otherwise $f\vee \widehat f^{\MLE}\in\Theta_n$ would have a larger likelihood than $\widehat f^{\MLE}$. Then $\int (f-f_0)_+ \leq \int (\widehat f^{\MLE}-f_0)_+$ holds such that with $A:=\{N: \int (\widehat f^{\MLE}-f_0)_+ > \eps_n\},$
\begin{align*}
	&E_{f_0}\Big[\Pi\Big( f\in \Theta_n : \int (f-f_0)_+ > \eps_n\Big | N \Big)\Big] \le  E_{f_0}\big[\Pi\big( f\in \Theta_n : A \big | N \big)\big]\\
	&= E_{f_0}\big[\Pi\big( f\in \Theta_n : A \big | N \big) \big(\mathbf{1}(A)+\mathbf{1}(A^c)\big)\big] \\
&= E_{f_0}[\mathbf{1}(A)]= P_{f_0} \Big(\int (\widehat  f^{\MLE}-f_0)_+ > \eps_n\Big),
\end{align*}
where the equalities hold because $A$ is independent of $f$.
%
%For the other bound the hypothesis
% \[\Pi\Big(f: \int(f_0-f) \leq A\eps_n, f\leq f_0\Big) \geq e^{-Cn\eps_n}\]
%implies (with the same proof as Prop. 1) for $B=\{f: \int(\widehat f^{\MLE}-f)_+\ge D\eps_n\}$
%\[ E_{f_0}[\Pi(B|N)]\le e^{(A+C)n\eps_n}\int_{\Theta_n} E_{f_0}[{\bf 1}(f\in B)\Lambda(f,f_0)]\Pi(df).\]
%By change of measure we have
%\[ E_{f_0}[{\bf 1}(f\in B)\Lambda(f,f_0)]=e^{-n\int(f_0-f)_+}P_{f_0\vee f}\Big(\int(\widehat f^{\MLE}-f_0\vee f)\ge D\eps_n-\int(f_0\vee f-f) \Big).\]
%For $f$ with $\int(f_0-f)_+\ge (A+C+1)\eps_n$ this is bounded by $e^{-(A+C+1)n\eps_n}$. For $f$ with $\int(f_0-f)_+< (A+C+1)\eps_n/2$ this is bounded by $P_{f_0\vee f}(\int(\widehat f^{\MLE}-f_0\vee f)\ge (D-A-C-1)\eps_n)$. Using that $\Theta_n$ is closed under maxima, we obtain the bound
%\[ E_{f_0}[\Pi(B|N)]\le e^{(A+C)n\eps_n}\int_{\Theta_n} \Big(e^{-(A+C+1)n\eps_n}+P_{f}\Big(\int(\widehat f^{\MLE}-f)\ge (D-A-C-1)\eps_n\Big) \Big)\,\Pi(df),\]
%as asserted.
%
%Interestingly, we can bound the posterior probability for $f$ deviating below $\widehat f^{\MLE}$ by an average bound for $\widehat f^{\MLE}$ deviating above $f$. Together with the first part this result seems to indicate that (under small ball property) the frequentist minimax risk (resp. contraction) of the Bayes method is never worse (in rate) than that of the MLE.
\end{proof}

\section{Proofs for Section \ref{sec.GP_priors}}

We state Theorem 2.1 of \cite{vvvz} in a slightly more general form.

\begin{thm}[Theorem 2.1 of \cite{vvvz}]
\label{thm.vvvz}
Let $X$ be a Borel-measurable, zero-mean Gaussian random element in the Banach space $(\B,\|\cdot\|)$ with RKHS $(\bbH, \|\cdot\|_{\bbH})$ and let $f$ be contained in the closure of $\bbH$ in $\B.$ For any $C_*\geq 1$ and all $\eps_n>0$, $\gamma_n \geq 1$, satisfying
$$\inf_{h:\|h-f\|\leq \eps_n} \|h\|_{\bbH}^2 - \log P(\|X\|\leq \eps_n) \leq \gamma_n,$$
there exists a Borel set $B_n \subset \B$ such that
\begin{align*}
	\log N(3\eps_n, B_n, \|\cdot \|)  &\leq 6C_*\gamma_n, \quad
	\P\big( X \notin B_n \big) \leq e^{-C_*\gamma_n}, \quad \text{and} \ \
	\P\big( \|X - f\| \leq 2\eps_n \big) \geq e^{-\gamma_n}.
\end{align*}
\end{thm}

\begin{proof}
Replace $n\eps_n^2$ in the proof of Theorem 2.1 of \cite{vvvz} by $\gamma_n;$ in particular $M_n := -2\Phi^{-1}(e^{-C_*\gamma_n}).$ For the final argument of the proof observe that $e^{-C_*\gamma_n} <1/2$ due to $C_*\gamma_n \geq 1.$
\end{proof}

\begin{rem}
\label{rem.vvvz_thm}
In the previous theorem, the condition that $f$ is contained in the closure of $\bbH$ in $\B$ can be avoided for null sequences $\eps_n\to 0$ if we agree that the infimum over the empty set is $+\infty.$
\end{rem}

%ideally one can also prove the result by substituting in Theorem 2.1 of van der Vaart and van Zanten $\eps_n$ for $\eps_n^2$ and $\|\cdot\|^{1/2}$ for $\|\cdot\|$. Notice, however, that $\|\cdot \|^{1/2}$ is not a norm anymore (does scale linear) and thus one cannot generate a Banach space.

\begin{proof}[Proof of Theorem \ref{thm.ub_for_GPs}]
We apply Theorem \ref{thm.vvvz} with $(\B, \|\cdot\|)=(\mC[0,1], \|\cdot\|_\infty),$ $\gamma_n = n\eps_n$, $f=f_0-2\eps_n$ and $C_*=6.$ This shows that there exists $\Theta_n$ such that $\log N(3\eps_n, \Theta_n, \|\cdot\|_\infty)\leq 36 n \eps_n,$ $P(X\notin\Theta_n) \leq e^{-6n\eps_n}$ and $P(\|X+2\eps_n -f_0\|_\infty \leq 2\eps_n) \geq e^{-n \eps_n}.$ Together with Remark \ref{rem.vvvz_thm}, the assumptions of Theorem \ref{thm.main_ub} are satisfied with $A=4, C=C''=1, C'=36$ in view of condition (ii') from \eqref{eq.condii_sup_norm} and condition \textit{(iii)} in Proposition \ref{prop.suffcond}.
\end{proof}

\section{Proofs for Section \ref{sec.Wavelet}}

\begin{proof}[Proof of Lemma \ref{lem.small_ball_prob_around_h}]
Write $h = \sum_{j,k} h_{j,k} \psi_{j,k}.$ Since $\psi$ is a compactly supported wavelet, $\|X-h \|_\infty \leq C \sum_{j} 2^{j/2} \max_k |d_{j,k} \xi_{j,k} -h_{j,k}|$ for a sufficiently large constant $C.$ By assumption $\psi$ is moreover $s$-regular and $h \in \mC^\beta(R)$ with $\beta \leq s.$ Using Theorem 4.4 in \cite{Cohen1993}, we can find constants $0< q<Q < \infty$ such that  $q2^{-\frac j2 (2\alpha+1)}\leq |d_{j,k}| \leq Q 2^{-\frac j2 (2\alpha+1)}$ and $|h_{j,k}|\leq Q 2^{-\frac j2 (2\beta+1)}$ and obtain for any $J,$
\begin{align*}
	\|X-h \|_\infty
	\leq CQ \Big( \sum_{j\leq J} 2^{-j\alpha} \max_k |\xi_{j,k} - h_{j,k}/d_{j,k}| + \sum_{j>J} 2^{-j\alpha} \max_k |\xi_{j,k}|
	+ \sum_{j>J} 2^{-j\beta} \Big).
\end{align*}
By assumption, there exists a $\delta>0,$ such that $L:=\E[|\xi_{j,k}|^{(1+\delta)/\alpha}]<\infty.$ Introduce the events
 \begin{align*}
 G_\le &=\{|\xi_{j,k} - h_{j,k}/d_{j,k}| \leq 2^{(j-J_*)2\alpha} \text{ for }j\leq J_*\text{ and all }k\},\\
  G_>&=\{|\xi_{j,k}|\leq L^{\alpha/(1+\delta)} 2^{(j-J_*)\alpha/(1+\delta/2)}\text{ for }j>J_*\text{ and all }k\},
\end{align*}
where $J_*$ is the smallest integer such that
\begin{align*}
	CQ\Big( 2^{-J_* \alpha} \sum_{r\geq 0} 2^{-\alpha r} + L^{\alpha/(1+\delta)} 2^{-J_* \alpha} \sum_{r\geq 1} 2^{-r\alpha \delta/(2+\delta)} +2^{-J_* \beta} \sum_{r\geq 1} 2^{-r\beta} \Big) \leq \eps,
\end{align*}
which yields $2^{J_*} \asymp \eps^{-1/(\alpha \wedge \beta)}$ as $\eps\rightarrow 0.$  On the event $G_\le$ we have $2^{-j\alpha} \max_k |\xi_{j,k} - h_{j,k}/d_{j,k}| \leq 2^{-j\alpha} 2^{(j-J_*)2\alpha} = 2^{-J_*\alpha}2^{(j-J_*)\alpha}$ and
on the event $G_>,$ $2^{-j\alpha} \max_k |\xi_{j,k}|\leq L^{\alpha/(1+\delta)} 2^{-J^*\alpha} 2^{-(j-J_*)\delta \alpha/(2+\delta)}.$ Then on $G_\le\cap G_>$, thanks to the choice of $J_*,$  we have $ \|X-h\|_\infty\leq \eps.$ Thus,
\begin{align}
	&\P(\|X-h\|_\infty \leq \eps \big)  \label{eq.sb_factorization}\\
	&\geq
	\prod_{j\leq J_* ,\, k} \P\big( |\xi_{j,k} - h_{j,k}/d_{j,k}| \leq 2^{(j-J_*)2\alpha} \big) \prod_{j>J_*,\,  k} \P\big( |\xi_{j,k}|\leq L^{\alpha/(1+\delta)} 2^{(j-J_*)\alpha/(1+\delta/2)}\big). \notag
\end{align}
On the event $\{|\xi_{j,k} - h_{j,k}/d_{j,k}| \leq 2^{(j-J_*)2\alpha}\}$  we  have for $j\leq J_*$ and $R':=1+q^{-1}Q$
\begin{align*}
	|\xi_{j,k} | \leq 2^{(j-J_*)2\alpha} + |h_{j,k}/d_{j,k}| \leq 2^{(j-J_*)2\alpha}+q^{-1} Q 2^{j(\alpha-\beta)}\leq R' 2^{J_*(\alpha-\beta)_+}.
\end{align*}
Since the random variables $\xi_{j,k}$ are symmetric and have a unimodal density, we have $f_\xi(x)\le 1/2$ for $x\ge 1$ as well as $\P\big( |\xi_{j,k} - h_{j,k}/d_{j,k}| \leq 2^{(j-J_*)2\alpha} \big)\geq 2^{(j-J_*)2\alpha} f_\xi(R' 2^{J_*(\alpha-\beta)_+}).$ On the $j$-th resolution level there are at most $A2^j$ wavelet coefficients with  some positive constant $A$. The first product in \eqref{eq.sb_factorization} can therefore be bounded from below by
\begin{align}
	\prod_{j\leq J_*} \Big( f_\xi(R' 2^{J_*(\alpha-\beta)_+}) 2^{(j-J_*)2\alpha}\Big)^{A2^j}
	&\geq f_\xi(R' 2^{J_*(\alpha-\beta)_+})^{A2^{J_*+1}} \prod_{r\le J_*} 2^{-2\alpha r2^{-r} A2^{J_*}} \notag \\
	&\geq  f_\xi(R' 2^{J_*(\alpha-\beta)_+})^{A2^{J_*+1}}  K^{-2^{J_*}}
	\label{eq.sb_factorization_p1}
\end{align}
for a sufficiently large constant $K.$ To find a lower bound of the second product in \eqref{eq.sb_factorization}, observe that by the moment bound on $\xi_{jk}$
\begin{align*}
	\P (|\xi_{j,k}| \leq L^{\alpha/(1+\delta)} 2^{(j-J_*) \alpha/(1+ \delta/2)})
	&= 1 - \P (|\xi_{j,k}|^{(1+\delta)/\alpha} > L 2^{(j-J_*) (1+\delta)/(1+ \delta/2)}) \\
	&\geq
	1 - 2^{(J_*-j) (1+\delta)/(1+ \delta/2)}.	
\end{align*}
For any fixed $j>J_*$ we use $ (1+\delta)/(1+ \delta/2)= 1 + \delta/(2+\delta)$ and the elementary inequality $1-y\ge e^{-2y}$, $0\leq y\leq 1/2,$ and obtain
\begin{align*}
	\prod_k \P (|\xi_{j,k}| \leq L 2^{(j-J_*) (1+\delta)/(1+ \delta/2)} )
	&\geq \big( 1 - 2^{(J_*-j) (1+\delta)/(1+ \delta/2)} \big)^{A2^j}\\
	&\geq \exp\big(-A 2^{J_*+1} 2^{(J_*-j)\delta/(2+\delta)} \big).
\end{align*}
This implies that the product $\prod_{j> J_*, \, k} \P (|\xi_{j,k}| \leq L^{\alpha/(1+\delta)} 2^{(j-J_*) \alpha/(1+ \delta/2)})$ can be bounded from below by
\begin{align}
	\prod_{j>J_*} \exp\big(-A2^{J_*+1} 2^{(J_*-j)\delta/(2+\delta)} \big)
	= \exp\big(-A2^{J_*+1} \sum_{k\geq 1} 2^{-k\delta/(2+\delta)}\big) \geq \exp(-R'' 2^{J_*})
	\label{eq.sb_factorization_p2}
\end{align}
for a sufficiently large constant $R''.$ Recall that $2^{J_*} \asymp \eps^{-1/(\alpha \wedge \beta)}.$  Because of $f_\xi(x)\le 1/2$ for $|x|\ge 1$ and $\eps \leq 1,$ we have for $K':= 1\vee (\log K+R'')/\log 2,$ $f_\xi (K' \eps^{- (\alpha-\beta)_+/\beta})^{K'}\leq 2^{-K'}\leq K^{-1}\exp(-R'')$ and raising both sides to the power $2^{J_*},$  $K^{-2^{J_*}}\exp(-R'' 2^{J_*}) \geq f_\xi (K' \eps^{- (\alpha-\beta)_+/\beta})^{K'2^{J_*}}.$ The result follows therefore from \eqref{eq.sb_factorization}, \eqref{eq.sb_factorization_p1}, and \eqref{eq.sb_factorization_p2}.
\end{proof}

% old lemma, replaced by the following one...

\begin{comment}
\begin{lem}
\label{lem.exp_ineq}
Suppose that for some $\gamma > 0,$ $f_\xi(x) \leq \gamma^{-1}e^{-\gamma |x|}$ for all $x\in \R.$ Let $\xi_1, \ldots, \xi_m \sim f_\xi,$ independently. Then,
\begin{align*}
	\P\Big(\frac1m\sum_{j=1}^m |\xi_j| \geq 2\gamma^{-1} (t+ \log (4/\gamma^2) ) \Big) \leq  e^{-t m}.
\end{align*}
\end{lem}

\begin{proof}
Set $A(\gamma, t) := 2\gamma^{-1} (t+ \log (4/\gamma^2) )$ and denote by $f_{|\xi|}$ the density of $|\xi|.$ By assumption, $f_{|\xi|}(x) \leq 2 \gamma^{-1}e^{-\gamma x}$ holds for all $x\geq 0$ such that $E[e^{\gamma |\xi|/2}] \leq 4/\gamma^2.$ Therefore, we deduce by Markov's inequality
\begin{align*}
	\P\Big(\sum_{j=1}^m |\xi_j| \geq A(\gamma, t) m\Big)
	&=
	\P\Big(\exp\Big( \frac \gamma 2 \sum_{j=1}^m |\xi_j| \Big)  \geq \exp\Big( \frac \gamma 2 A(\gamma, t) m\Big) \Big)\\
	&\leq \exp \Big( m\log\Big( \frac{4}{\gamma^2} \Big) - \frac \gamma 2 A(\gamma, t) m\Big) = e^{-tm}.
\end{align*}
\end{proof}
\end{comment}

\begin{lem}
\label{lem.exp_ineq_new}
\mbox{}
\begin{enumerate}
\item Let $Z_1, \ldots,Z_m$ be i.i.d. random variables with $\E[e^{Z_1/K}]\le e^{r/K}$ for some $K,r>0$. Then,
\begin{align*}
	\P\Big(\frac1m\sum_{j=1}^m Z_j \geq r+t  \Big) \le   e^{-mt/K},\quad t\ge 0.
\end{align*}
\item Let $Z_1, \ldots,Z_m$ be i.i.d. random variables with $R:=\E[\abs{Z_1}e^{\abs{Z_1}^\gamma/\kappa}]<\infty$ for some $\gamma\in(0,1],\kappa>0$. Then for some $C>0$ and all $m\in\N$
\[ \P\Big(\frac1m\sum_{j=1}^m \abs{Z_j} \geq R+t \Big) \le Ce^{-(tm)^\gamma/\kappa},\quad t\ge 1.\]
\end{enumerate}
\end{lem}

\begin{proof} The first inequality follows directly from the exponential Markov inequality:
\[ \P\Big(\sum_{j=1}^m Z_j \geq (t+r)m  \Big) \leq \E[\exp(Z_1/K)]^m e^{-(t+r)m/K}\le e^{-tm/K}.\]
To show the second assertion, consider the truncated random variables $V_{j,m}:=\abs{Z_j}{\bf 1}(\abs{Z_j}\le tm)$. Observe that $e^x\leq 1+xe^x$ for positive $x$  and let $K_m=\kappa (tm)^{1-\gamma}.$ Together with $V_{j,m}\leq \abs{Z_j}$ and $V_{j,m}=V_{j,m}^\gamma V_{j,m}^{1-\gamma}\leq \abs{Z_j}^\gamma K_m/\kappa,$ we find
\begin{align*}
	\E\big[e^{V_{j,m}/K_m}\big]
	&\leq 1 +K_m^{-1} \E\big[V_{j,m} e^{V_{j,m}/K_m}\big]
	\leq  1 + \frac{R}{K_m}
	\leq \exp(R/K_m).
\end{align*}
From the first part we thus derive
\[ \P\Big(\frac1m\sum_{j=1}^m V_{j,m} \geq R+t  \Big) \le   e^{-mt/K_m}=e^{-(tm)^\gamma /\kappa }.\]
On the other hand, we estimate by the union bound and Markov's inequality
\begin{align*}
\P(\exists j=1,\ldots,m: \abs{Z_j}>tm)\le m\E[\abs{Z_1}e^{\abs{Z_1}^\gamma/\kappa}](tm)^{-1}e^{-(tm)^\gamma/\kappa}= R e^{-(tm)^\gamma/\kappa}/t.
\end{align*}
Taking the two deviations bounds together, we deduce the result for $t\ge 1$.
\end{proof}

\begin{proof}[Proof of Theorem \ref{thm.contr_rates_for_wav_priors}]
It is enough to prove the result for all $n \geq n_0$ with $n_0$ a fixed integer. We verify the conditions $(i)-(iii)$ of Theorem \ref{thm.main_ub}, starting with condition $(ii).$

{\it (ii):} Apply Lemma \ref{lem.small_ball_prob_around_h} to \eqref{eq.condii_sup_norm}. Since $\eps_n \rightarrow 0,$   $f_0 - \eps_n  \in \mC^\beta(R+1)$ for sufficiently large $n$ and we deduce by assumption on $\eps_n$
\begin{align}
	\log\Big(\Pi \big(  f : \| f + \eps_n - f_0 \|_\infty \leq \eps_n\big)\Big)
	\geq D\eps_n^{-1/(\alpha \wedge \beta)}\log\Big(f_\xi\big(D\eps_n^{-(\alpha -\beta)_+/\beta}\big)\Big)
	\gtrsim - n \eps_n.	
		\label{eq.cond_ii_checked}
\end{align}

{\it (i):} We first need to identify a $\Theta_n$ which covers most of the prior mass and has small metric entropy.  Besov spaces provide a natural framework to study wavelet decay. It turns out, however, that the low resolution levels and the bias part should be embedded into different Besov balls. For a level $J$ and some constants $p>\alpha^{-1}$ and $K>0$, which will  be chosen later, define
\begin{align*}
	\Theta_n = \Big \{ g= \sum_{j,k}\theta_{j,k}\psi_{j,k} : &\sum_{j \leq J} 2^{jp(\alpha + 1/2)} \sum_k |\theta_{j,k}|^p \leq K^p 2^{J}, \\
	&\max_{j>J} 2^{jp (\alpha +1/2-1/p)}\sum_{k} |\theta_{j,k}|^p \leq K^p\Big\}.
\end{align*}
Denote by
\begin{align*}
	B_{p,q}^s(M):= \Big\{g= \sum_{j,k}\theta_{j,k}\psi_{j,k} :  \Big(\sum_j 2^{qj (s+ 1/2 - 1/p)} \big(\sum_k |\theta_{j,k}|^p \big)^{q/p}\Big)^{1/q} \leq M\Big\}	
\end{align*}
the Besov $B_{p,q}^s$-ball with radius $M$ and apply the usual modifications for $p =\infty$ and $q=\infty.$ For a reference see for instance \cite{GineNickl}, page 325. To bound the bracketing entropy of $\Theta_n,$ observe that $\Theta_n \subseteq  B_{p,p}^{\alpha +1/p}(K2^{J/p}) +  B_{p, \infty}^\alpha(K)$ where the sum is the elementwise addition. By the metric entropy bounds of Theorem 4.3.36 in \cite{GineNickl}, extended to the  more general quasi-Banach space setting $p,q>0$ following \cite{MR969551} or \cite{mayer2019entropy}, there exists a constant $C'$ such that $\log \mathcal{N}(\delta, B_{p,q}^s(M), \|\cdot\|_\infty )\leq C' (M/\delta)^{1/s}$ if $s>1/p$ for any $p, q>0$.  In view of $\alpha>1/p$  the metric entropy bounds give
\begin{align*}
	&\log \mathcal{N}\big( \eps_n , \Theta_n, \|\cdot\|_\infty \big) \notag \\
	&\leq \log \mathcal{N}\big(\eps_n/2, B_{p,p}^{\alpha +1/p}(K2^{J/p}), \|\cdot\|_\infty \big)
	+ \log \mathcal{N}\big( \eps_n/2,  B_{p,\infty}^{\alpha}(K), \|\cdot\|_\infty \big)
\\
	&\lesssim 2^{J/(\alpha p+1)}\eps_n^{-1/(\alpha+1/p)}+\eps_n^{-1/\alpha}. \notag
\end{align*}
Property {\it (i)} of Theorem \ref{thm.main_ub} is therefore satisfied if
\begin{align}
	 2^{J/(\alpha p+1)}\eps_n^{-1/(\alpha+1/p)} + \eps_n^{-1/\alpha}\lesssim n\eps_n,\quad \alpha>1/p.
	\label{eq.ent1}
\end{align}

{\it (iii):} We bound $\Pi(\Theta_n^c).$ Recall that $X= \sum_{j,k} d_{j,k} \xi_{j,k} \psi_{j,k}$ and $ |d_{j,k}| \leq Q 2^{-j(\alpha+1/2)}$ for all $j,k.$ Thus,
\begin{align*}
	\Pi\big(\Theta_n^c \big)
	&\leq
	\P \Big( \sum_{j \leq J} 2^{jp(\alpha + 1/2)} \sum_k | d_{j,k} \xi_{j,k}|^p > K^p2^{J} \Big)
	+ \P \Big( \max_{j>J} 2^{jp (\alpha +1/2-1/p)}\sum_{k} |d_{j,k} \xi_{j,k}|^p > K^p\Big) \\
	&\leq
	\P \Big( Q^p2^{-J}\sum_{j \leq J}  \sum_k | \xi_{j,k}|^p >  K^p \Big)+
	\sum_{j>J} \P \Big( Q^p 2^{-j} \sum_k |\xi_{j,k} |^p > K^p \Big).
\end{align*}
On the $j$-th resolution level there are of the order of $2^j$ wavelet coefficients. Since by assumption $f_\xi(x) \leq \gamma^{-1}e^{-\gamma |x|^q}$, we have that $E[|\xi_{j,k}| \exp(\gamma |\xi_{j,k}|^q/2)]< \infty$. Thus,  the large deviations bound in Lemma \ref{lem.exp_ineq_new} with $\gamma=q/p$ and $m\asymp 2^J$ shows  that for any constant $c>0$ there is some choice of $K>0$ such that
\[\Pi(\Theta_n^c) \lesssim \exp(-c2^{Jq/p}).\]
This means that {\it (iii)} of Theorem \ref{thm.main_ub} is satisfied if there are constants $J$ and $c$ such that
\begin{align}
	c2^{Jq/p}\ge n\eps_n, \quad \text{for some} \ q \ \text{with} \ \alpha>1/(2q)\text{ for } q\le p.
	\label{eq.ent2}
\end{align}

Both \eqref{eq.ent1} and \eqref{eq.ent2} hold for some $c,C$ if we choose $2^J \asymp (n\eps_n)^{p/q}$ and if \begin{equation}\label{eq.condfinal}
(n\eps_n)^{1/q}\lesssim (n\eps_n)^{\alpha +1/p}\eps_n\text{ and }\eps_n \geq n^{-\alpha/(\alpha+1)},
\end{equation}
provided $\alpha>1/p$, $q\le p$.
In the case $q>1/\alpha$, we choose $p=q$ and \eqref{eq.condfinal} is satisfied for $\eps_n\gtrsim n^{-\alpha/(\alpha+1)}$. For
$q\in((2\alpha)^{-1},\alpha^{-1}]$ we can choose any $p>1/\alpha$ such that $\alpha+1/p-1/q>0$ and  \eqref{eq.condfinal} is satisfied for $\eps_n\gtrsim n^{-(\alpha+1/p-1/q)/(\alpha+1/p-1/q+1)}$. For a given $\delta>0$ we take $p$ sufficiently close to $1/\alpha$ such that $n^{-(\alpha+1/p-1/q)/(\alpha+1/p-1/q+1)}\le n^{-(2\alpha-1/q)/(2\alpha-1/q+1)+\delta}$. This yields the claim.
\end{proof}

\begin{rem}\label{rem.on_bias}
Observe that the condition $p>1/\alpha$ in the proof for $(ii)$ is due to the bias part. This condition leads to a slower rate in the case $q<1/\alpha$ as it forbids to choose $p=q.$ In nonparametric Bayes, wavelet priors with uniform $\xi_{j,k}$ have frequently been considered which allow for a particularly simple analysis of the high-resolution levels. Indeed for uniform priors on the wavelet coefficients, $p>1/\alpha$ always holds and one can even argue directly using the deterministic bound $\| \sum_{j>J,k} d_{j,k}\xi_{j,k}\psi_{j,k}\|_\infty \lesssim \sum_{j>J,k} 2^{-j\alpha}\lesssim 2^{-J\alpha}.$ Therefore, $\| \sum_{j>J,k} d_{j,k}\xi_{j,k}\psi_{j,k}\|_\infty\leq \eps_n/2$ for the standard choice $2^{-J\alpha}\asymp \eps_n$ and a correct adjustment of constants. Thus, to verify the conditions $(ii)$ and $(iii)$ of Theorem \ref{thm.main_ub} in this case, it is enough to cover a subset $ \Theta_n'$ of the space $\{t\mapsto \sum_{j\leq J,k} d_{j,k}a_{j,k}\psi_{j,k}(t): a_{j,k} \in \R\}$ using at most $e^{Cn\eps}$ balls of radius $\eps_n/2$ and show that $\Pi( (\Theta_n')^c)\leq e^{-cn\eps_n}$ for sufficiently large $c.$
\end{rem}

\begin{proof}[Proof of Lemma \ref{lem.small_ball_prob_around_h_trunc_prior}]
Since $\psi$ is $s$-regular and $\beta \leq s,$ we have $|h_{j,k}|\lesssim 2^{-j (\beta +1/2)}.$ As $\psi$ has compact support, there exists a constant $C$ such that
\[\|X-h\|_\infty \leq C \Big(\sum_{j \leq J} 2^{j/2}\max_k |\xi_{j,k} - h_{j,k}| +2^{-J\beta}\Big).\]
Let $J^*$ be the smallest integer $J$ such that $C (\sum_{j \leq J} \eps/(2JC)+2^{-J\beta}) \leq \eps,$ implying $2^{-J^*\beta}\lesssim\eps$. Then
\begin{align*}
	\P(\|X-h\|_\infty \leq \eps )
	\geq
	\P(J=J^*) \prod_{j \leq J^*,\, k}	 \P \big( 2^{j/2}|\xi_{j,k} - h_{j,k}| \leq \eps/(2J^*C) \big).
\end{align*}
By construction, $2^{j/2}\abs{h_{jk}}+\eps/(2J^*C)$ is uniformly bounded over all $j,k$ and $\eps\in(0,1]$ such that for the positive continuous density $f_\xi$ we find a uniform constant $c>0$ with
\[ \forall j\le J^*,k:\;\P \big( 2^{j/2}|\xi_{j,k} - h_{j,k}| \leq \eps/(2J^*C) \big)\ge c 2^{-J^*/2}\eps/J^*.\]
Together with $\P(J=J^*) \propto \exp(-BJ^*2^{J^*})$ and the fact that on the $j$-th resolution level the number of wavelet coefficients is bounded by $A2^j$ for some $A>0,$ this shows that
\begin{align*}
	\P(\|X-h\|_\infty\leq \eps )
	\gtrsim \exp(-BJ^*2^{J^*}) \Big(2^{-J^*/2}\eps/J^*\Big)^{A2^{J^*+1}}
\end{align*}
and thus with $\eps\asymp 2^{-J^*\beta}$ also $\log(\P(\|X-h\|_\infty\leq \eps ))\gtrsim -J^*2^{J^*}\asymp -\eps^{-1/\beta}\log(\eps^{-1})$, which was to be shown.
\end{proof}

\begin{proof}[Proof of Theorem \ref{thm.contr_rates_for_trunc_wav_priors}]
We verify the conditions $(i)-(iii)$ of Theorem \ref{thm.main_ub}.

{\it (i):} To check the first condition, pick the largest $J_n$ such that $J_n2^{J_n}\le c'n\eps_n$ for some small constant $c'>0$ to be chosen later. For a constant $K$, which will be chosen later to be large enough, define
\begin{align*}
	\Theta_n = \Big \{ g= \sum_{j \leq J_n, \, k}\theta_{j,k}\psi_{j,k} :  \sum_{j \leq J_n, \, k} |\theta_{j,k}|^q \leq K^q J_n 2^{J_n}\Big\}.
\end{align*}
Since there are at most $A2^{J_n}$ many non-zero wavelet coefficients in $\Theta_n$ for some constant $A$, we just need a covering in $\R^{A2^{J_n}}$.  By a classical entropy bound, see \cite{MR969551} or Theorem 4.3.35 in \cite{GineNickl} (whose proof also covers the case $q\in(0,1)$), we have
\[\log \mathcal{N}\big(  \eps_n , \Theta_n, \|\cdot\|_\infty \big)\lesssim A2^{J_n}\log\Big(2/\Big(A2^{J_n}(\eps_n/K)^qJ_n^{-1}2^{-3J_n/2}\Big)\Big)\lesssim (J_n+\log(\eps_n^{-1}))2^{J_n}.\]
%By the previous arguments in the proof of Theorem \ref{thm.contr_rates_for_wav_priors} we obtain via the Besov ball inclusion $\Theta_n\subset B^s_{q,1}(KJ_n^{1/q}2^{J_n(s+1/2)})$ for any $s>1/q$
%\begin{align*}
%	\log \mathcal{N}\big(  C\eps_n , \Theta_n, \|\cdot\|_\infty \big)
%	&\lesssim \Big(\frac{2^{J_n(s+1/2)}}{C\eps_n}\Big)^{1/s}\asymp n\eps_n\big(n/(C\eps_n)\big)^{1/s}.
%\end{align*}
Since $\log(\eps_n^{-1})\lesssim \log n\lesssim J_n$ we obtain $\mathcal{N}\big(  \eps_n , \Theta_n, \|\cdot\|_\infty \big) \leq C'' e^{C' n\eps_n}$ for some finite constant $C', C''.$

{\it (ii):} Since $\eps_n \rightarrow 0,$   $f_0 - \eps_n  \in \mC^\beta(R+1)$ for sufficiently large $n$. The result follows from applying Lemma \ref{lem.small_ball_prob_around_h_trunc_prior} to \eqref{eq.condii_sup_norm} and
\begin{align*}
	\Pi \big(  f : \| f + \eps_n - f_0 \|_\infty \leq \eps_n\big)
	\geq \eps_n^{D \eps_n^{-1/\beta}}
	\geq e^{-C n \eps_n},
\end{align*}
for sufficiently large $C.$

{\it (iii):} Observe that $\Pi\big(\Theta_n^c \big)\leq  \P(J > J_n) + \P ( \sum_{j \leq J_n, \, k} | \xi_{j,k}|^q \geq K^qJ_n2^{J_n}).$ The sum $\sum_{j \leq J_n, \, k}$ is over $A2^{J_n}$ wavelet coefficients. Recalling $J_n2^{J_n} \asymp n\eps_n$, Lemma \ref{lem.exp_ineq_new}(1) with $Z_1\sim\abs{\xi_{jk}}^q$  shows that we obtain $\Pi(\Theta_n^c) \leq e^{- cn\eps_n}$ for a constant $c>0$ that can be made as large as needed by increasing $K.$

The assertion follows  from Theorem \ref{thm.main_ub}.
\end{proof}

\bibliographystyle{acm}       % (uses file "plain.bst")
\bibliography{bibIM}           % expects file "refsPart1.bib"

\begin{thebibliography}{10}

\bibitem{bochkina2014}
{\sc Bochkina, N.~A., and Green, P.~J.}
\newblock {The Bernstein - von Mises theorem and nonregular models}.
\newblock {\em Ann. Statist. 42}, 5 (2014), 1850--1878.

\bibitem{chernozhukov2004}
{\sc Chernozhukov, V., and Hong, H.}
\newblock Likelihood estimation and inference in a class of nonregular
  econometric models.
\newblock {\em Econometrica 72}, 5 (2004), 1445--1480.

\bibitem{Cohen1993}
{\sc Cohen, A., Daubechies, I., and Vial, P.}
\newblock Wavelets on the interval and fast wavelet transforms.
\newblock {\em Appl. Comput. Harmon. Anal. 1}, 1 (1993), 54--81.

\bibitem{MR969551}
{\sc Edmunds, D.~E., and Triebel, H.}
\newblock Entropy numbers and approximation numbers in function spaces.
\newblock {\em Proc. London Math. Soc. (3) 58}, 1 (1989), 137--152.

\bibitem{ggv}
{\sc Ghosal, S., Ghosh, J.~K., and van~der Vaart, A.~W.}
\newblock Convergence rates of posterior distributions.
\newblock {\em Ann. Statist. 28}, 2 (2000), 500--531.

\bibitem{ghosal2007}
{\sc Ghosal, S., and van~der Vaart, A.}
\newblock Convergence rates of posterior distributions for non-i.i.d.
  observations.
\newblock {\em Ann. Statist. 35}, 1 (2007), 192--223.

\bibitem{ghoshal2017}
{\sc Ghosal, S., and van~der Vaart, A.~W.}
\newblock {\em {Fundamentals of Nonparametric Bayesian Inference}}.
\newblock Cambridge University Press, Cambridge, 2017.

\bibitem{Gine2011}
{\sc Gin\'e, E., and Nickl, R.}
\newblock Rates on contraction for posterior distributions in {$L^r$}-metrics,
  {$1\leq r\leq\infty$}.
\newblock {\em Ann. Statist. 39}, 6 (2011), 2883--2911.

\bibitem{GineNickl}
{\sc Gin\'e, E., and Nickl, R.}
\newblock {\em Mathematical Foundations of Infinite-Dimensional Statistical
  Models}.
\newblock Cambridge University Press, Cambridge, 2016.

\bibitem{hoffmann2015}
{\sc Hoffmann, M., Rousseau, J., and Schmidt-Hieber, J.}
\newblock On adaptive posterior concentration rates.
\newblock {\em Ann. Statist. 43}, 5 (2015), 2259--2295.

\bibitem{jirak2014}
{\sc Jirak, M., Meister, A., and Rei{\ss}, M.}
\newblock Adaptive function estimation in nonparametric regression with
  one-sided errors.
\newblock {\em Ann. Statist. 42}, 5 (2014), 1970--2002.

\bibitem{2012arXiv1210.6204K}
{\sc {Kleijn}, B., and {Knapik}, B.}
\newblock {Semiparametric posterior limits under local asymptotic
  exponentiality}.
\newblock {\em ArXiv e-prints\/} (2012).

\bibitem{kutoyants1998}
{\sc Kutoyants, Y.}
\newblock {\em Statistical Inference for Spatial Poisson Processes}.
\newblock Springer, 1998.

\bibitem{LeCam1990}
{\sc Le~Cam, L., and Yang, G.~L.}
\newblock {\em Asymptotics in statistics}.
\newblock Springer Series in Statistics. Springer-Verlag, New York, 1990.

\bibitem{LiGhoshal2015}
{\sc {Li}, M., and {Ghosal}, S.}
\newblock {Bayesian detection of image boundaries}.
\newblock {\em Ann. Statist. 45}, 5 (2017), 2190--2217.

\bibitem{Lo1982}
{\sc Lo, A.~Y.}
\newblock {Bayesian nonparametric statistical inference for Poisson point
  processes}.
\newblock {\em Zeitschrift f{\"u}r Wahrscheinlichkeitstheorie und Verwandte
  Gebiete 59}, 1 (1982), 55--66.

\bibitem{mayer2019entropy}
{\sc Mayer, S., and Ullrich, T.}
\newblock Entropy numbers of finite dimensional mixed-norm balls and function
  space embeddings with small mixed smoothness.
\newblock Tech. rep., arXiv eprint 1904.04619, 2019.

\bibitem{meisterreiss2013}
{\sc Meister, A., and Rei\ss, M.}
\newblock Asymptotic equivalence for nonparametric regression with non-regular
  errors.
\newblock {\em Probab. Theory Related Fields 155}, 1-2 (2013), 201--229.

\bibitem{Ray2013}
{\sc Ray, K.}
\newblock Bayesian inverse problems with non-conjugate priors.
\newblock {\em Electron. J. Stat. 7\/} (2013), 2516--2549.

\bibitem{reiss2018b}
{\sc {Rei\ss }, M., and {Schmidt-Hieber}, J.}
\newblock {Nonparametric Bayesian analysis of the compound Poisson prior for
  support boundary recovery}.
\newblock {\em Annals of Statistics, to appear\/} (2019).

\bibitem{reiss2014}
{\sc Rei{\ss}, M., and Selk, L.}
\newblock Efficient estimation of functionals in nonparametric boundary models.
\newblock {\em Bernoulli 23}, 2 (2017), 1022--1055.

\bibitem{vvvz}
{\sc van~der Vaart, A.~W., and van Zanten, H.}
\newblock {Rates of contraction of posterior distributions based on Gaussian
  process priors}.
\newblock {\em Ann. Statist. 36}, 3 (2008), 1435--1463.

\bibitem{wang1997}
{\sc Wang, Y.}
\newblock {Small ball problem via wavelets for Gaussian processes}.
\newblock {\em Statist. Probab. Lett. 32}, 2 (1997), 133 -- 139.

\end{thebibliography}

\end{document}